\documentclass{amsart}
\usepackage{amsmath}

\newtheorem{theorem}{Theorem}[section]

\newtheorem{lemma}[theorem]{Lemma}

\theoremstyle{definition}

\title[Volume-surface reaction-diffusion] 
      {Global existence and uniqueness for a volume-surface reaction-nonlinear-diffusion
system}

\author[K. Disser]{}


 \email{kdisser@mathematik.tu-darmstadt.de}

\begin{document}
\maketitle

\centerline{\scshape Karoline Disser$^*$}
\medskip
{\footnotesize
 \centerline{FB Mathematik, TU Darmstadt,}
   \centerline{ Schlossgartenstr. 7,}
   \centerline{ 64293 Darmstadt, Germany}
} 



\bigskip


\begin{abstract}
We prove a global existence, uniqueness and regularity result for a two-species reaction-diffusion
volume-surface system that includes nonlinear bulk diffusion and nonlinear
(weak) cross diffusion on the active surface. A key feature is a proof
of upper $L^{\infty}$-bounds that exploits the entropic gradient
structure of the system.
\end{abstract}

\section{Introduction}

There has been a lot of recent interest in the mathematical analysis
of reaction-diffusion processes that are coupled across volume and
surface domains, due to their relevance in particular in the modelling
of biological cells and technological processes, e.g. involving
catalysis by surfactants \cite{BKMS17, BFL18, FKMMN16, Glitzky12}. For systems coupled across volume
and surface, a particular modelling issue and mathematical difficulty
is the nonlinearity of diffusion that is typical of many biological
processes (due to porous media, non-Newtonian flow, ...) and naturally
associated to diffusion that is modelled on multiple (spatial) scales
(\cite{Bothe15, GM13, Keil13, KB08, Mielke13}). The aim of this paper is to prove a global well-posedness result
for reaction-diffusion volume-surface systems that include typical
general classes of nonlinear diffusion (like slow and fast diffusion)
and nonlinear (weak) cross diffusion on the active surface. This generalizes
some of the recent results on volume-surface reaction-diffusion \cite{BKMS17, BFL18, FKMMN16} to the nonlinear diffusion case and to the case that only part
of the boundary of the bulk domain is active. 

Besides the well-posedness result based on the analysis for quasilinear volume-surface systems in \cite{Disser17}, the aim of the paper is to introduce a proof of $L^{\infty}$-estimates that is different from the
classical proof of comparison principles for the linearized part \cite{BFL18, FKMMN16}. Instead, we explicitly exploit the entropic gradient
structure of the system \cite{GM13} to show that the (chemical) logarithmic potentials remain pointwise bounded. This translates to
pointwise estimates on the concentrations. A feature of
this estimate is that it does not use the positivity of the (quasi)linear
diffusion.  The method was inspired by the boundedness-by-entropy
method developed by J\"ungel \cite{Juengel15}, but the argument here is more direct
and thus does not promise as much generality. So it is unclear whether
a strategy of this type can also be used to obtain results on more
complex cases of multicomponent diffusion or reaction. This is the
aim of future work. 

\subsection{Model equations}\label{model}
We consider the reaction-cross-diffusion system 
\begin{align}
\dot{u}-\mathrm{div}(\mu(u)\nabla u) & =0,\qquad\text{in }(0,\infty)\times\Omega,\nonumber \\
\mu(u)\frac{\partial u}{\partial\nu} & =0,\qquad\text{on }(0,\infty)\times\partial\Omega\setminus\Gamma,\nonumber \\
\mu(u)\frac{\partial u}{\partial\nu}+k\alpha(u^{\alpha}-\kappa v^{\beta}) & =0,\qquad\text{on }(0,\infty)\times\Gamma,\label{eq:System1}\\
\dot{v}-\mathrm{div}_{\Gamma}(\mu_{\Gamma}(u,v)\nabla_{\Gamma}v)-k\beta(u^{\alpha}-\kappa v^{\beta}) & =0,\qquad\text{on }(0,\infty)\times\Gamma,\nonumber \\
\mu_{\Gamma}(u,v)\frac{\partial v}{\partial\nu_{\Gamma}} & =0,\qquad\text{on }(0,\infty)\times\partial\Gamma,\nonumber  
\end{align}
with initial conditions $u(0,x)=u_{0}(x)>0,x\in\Omega$ and $v(0,y)=v_{0}(y)>0,y\in\Gamma.$
Here, $\Omega\subset\mathbb{R}^{d}$,$d=2,3$ is a bounded strong
Lipschitz domain, $\Gamma\subset\partial\Omega$ is an open connected part of its
boundary, $\nu$ is the unit outer normal vector field at $\partial\Omega$
and if $d=3$, $\nu_{\Gamma}$ is the unit outer normal vector field
at $\partial\Gamma$ (for a (canonical) definition and discussion
of the operators $\mathrm{div}_{\Gamma}$ and $\nabla_{\Gamma}$ in
this non-smooth case, see e.g. \cite{DMR15}). Locally uniformly
in $(u,v)\in\mathbb{R}_{+}\times\mathbb{R}_{+}$, the diffusion coefficients
$\mu(\cdot,u)$, $\mu_{\Gamma}(\cdot,u,v)$ are measurable, bounded
and elliptic. Uniformly in $x\in\Omega,y\in\Gamma$, $u\mapsto\mu(x,u)$
and $(u,v)\mapsto\mu_{\Gamma}(y,u,v)$ are locally Lipschitz. Examples are slow and fast diffusion like
\[
\mu(x,u)=u^{\gamma}\mathrm{Id}_{d},\quad\mu(x,u)=e^{\delta u}\mathrm{Id}_{d},\quad\gamma,\delta\in\mathbb{R}.
\]
Correspondingly,
these and expressions like
\[
\mu_{\Gamma}(u,v)=\frac{v}{\alpha u+\beta v}\mathrm{Id}_{d-1},
\]
are admissible for the diffusion coefficient on the surface. With natural boundary conditions,
starting from positively bounded initial values, no degeneracy will
develop for this system. The reversible chemical reaction on the surface
is described by mass-action kinetics with forward and backward rates
$k,\kappa>0$ and stoichiometric coefficients $\alpha,\beta\geq1$. 

\subsection{Outline}
Section \ref{sec:Preliminary-results-and} on preliminaries is mostly technical, in that
it is used to introduce the notation and arguments needed to fit
system (\ref{eq:System1}) in the framework of the main result in
\cite{Disser17} on existence and uniqueness for quasilinear dissipative bulk-interface
dynamics. It comes before the main result as it is needed for the
detailed/precise statement of the main result in Theorem \ref{thm:mr} in Section \ref{sec:Main-result}.
The $L^{\infty}$-estimate derived from the entropic gradient
structure of system (\ref{eq:System1}) is shown in Section \ref{sec:-estimates-based-on}. 

\section{\label{sec:Preliminary-results-and}Preliminary results and notation}

Due to the quasilinear structure of system (\ref{eq:System1}),
global existence and uniqueness do not follow from standard linearization
techniques, if only $L^{\infty}$ bounds are available
globally a priori. To see that system (\ref{eq:System1}) fits
with the theory in \cite{Disser17}, for convenience and in order to give a precise statement of
the main result, we introduce some of the functional analytic framework used for volume-surface systems
in \cite{Disser17}. 

\subsection{Notation and function spaces}
For $q,q_{\Gamma}\in[1,\infty]$ and $0\leq\alpha,\alpha_{\Gamma}<1$,
define 
\begin{align*}
\mathrm{L}^{q,q_{\Gamma}} & :=L^{q}(\Omega)\times L^{q}(\Gamma),\\
\mathrm{W}^{1,q,q_{\Gamma}} & :=W^{1,q}(\Omega)\times W^{1,q}(\Gamma),\\
\mathrm{C}^{\alpha,\alpha_{\Gamma}} & :=C^{\alpha}(\Omega)\times C^{\alpha_{\Gamma}}(\Gamma),
\end{align*}
with components $\mathrm{L}^{q,q_{\Gamma}}\ni\varphi=(\varphi_{u},\varphi_{v})$,
where $L^{q}$, $W^{1,q}$, $C^{\alpha}$ denote the usual Lebesgue,
Sobolev and Hölder spaces. In the following, $q'$ is the dual exponent
of $q$ with $1=\frac{1}{q}+\frac{1}{q'}$ and we denote dual Sobolev
spaces by
\[
W^{-1,q}(\Omega)=\left(W^{1,q'}(\Omega)\right)',W^{-1,q_{\Gamma}}(\Gamma)=\left(W^{1,q'_{\Gamma}}(\Gamma)\right)'
\]
and 
\[
\mathrm{W}^{-1,q,q_{\Gamma}}:=\left(\mathrm{W}^{1,q',q'_{\Gamma}}\right)'.
\]
Let $(u,v)\in\mathrm{C}^{0,0}$ with $u,v>0$. Then using the assumptions
on $\mu,\mu_{\Gamma}$, the bilinear form $\mathfrak{l}_{u,v}\colon\mathrm{W}^{1,2,2}\times\mathrm{W}^{1,2,2}\to\mathbb{R}$
given by 
\[
\mathfrak{l}_{u,v}(\psi,\varphi):=\int_{\Omega}\nabla\psi_{u}\cdot\mu(u)\nabla\varphi_{u}\,\mathrm{d}x+\int_{\Gamma}\nabla_{\Gamma}\psi_{v}\cdot\mu(u,v)\nabla\varphi_{v}\,\mathrm{d}y
\]
is well-defined, continuous and coercive (after shifting). It induces
the two-component Neumann Laplacian $\mathcal{L}_{u,v}\colon\mathrm{W}^{1,2,2}\to\mathrm{W}^{-1,2,2}$
with $\mathcal{L}_{u,v}(\psi)(\varphi):=\mathfrak{l}_{u,v}(\psi,\varphi)$.
For $q,q_{\Gamma}\in[2,\infty)$, let $\mathcal{L}_{u,v}^{q,q_{\Gamma}}$
be the closed and densely defined restriction of $\mathcal{L}_{u,v}$
to $\mathrm{W}^{-1,q,q_{\Gamma}}$. 

\subsection{Interpretation of (\ref{eq:System1})}
In order to re-write system (\ref{eq:System1})
as a quasilinear abstract Cauchy problem in $\mathrm{W}^{-1,q,q_{\Gamma}}$,
consider the reaction terms as components of a functional $\mathcal{F}(u,v)\in\mathrm{W}^{-1,q,q_{\Gamma}}$
through
\[
\mathcal{F}(u,v)(\varphi)=\left(\begin{array}{l}
\mathcal{F}_{u}(u,v)(\varphi)\\
\mathcal{F}_{v}(u,v)(\varphi)
\end{array}\right)=\left(\begin{array}{l}
-\alpha\int_{\Gamma}k(u^{\alpha}-\kappa v^{\beta})\varphi_{u}|_{\Gamma}\,\mathrm{d}y\\
+\beta\int_{\Gamma}k(u^{\alpha}-\kappa v^{\beta})\varphi_{v}\,\mathrm{d}y
\end{array}\right),
\]
for all $\varphi\in\mathrm{W}^{1,q',q_{\Gamma}'}$. With the existence
of the traces $u|_{\Gamma}\in C^{0}(\Gamma)$ and $\varphi_{u}|_{\Gamma}\in L^{q'}(\Gamma)$,
$\mathcal{F}(u,v)$ is well-defined for all $t>0$. Note that this
is just the usual way of casting inhomogeneous Neumann boundary conditions
in a weak form. Hence, system (\ref{eq:System1}) can be written as
the abstract Cauchy problem 
\begin{equation}
\left(\begin{array}{c}
\dot{u}\\
\dot{v}
\end{array}\right)(t)+\mathcal{L}_{u(t),v(t)}^{q,q_{\Gamma}}(u(t),v(t))=\mathcal{F}(u(t),v(t))\label{eq:System2}
\end{equation}
in $\mathrm{W}^{-1,q,q_{\Gamma}}$ with initial data $(u,v)(0)=(u_{0},v_{0})$.
Due to the uniqueness and regularity of solutions (Theorem \ref{thm:mr}),
solutions of (\ref{eq:System2}) solve system (\ref{eq:System1})
in an adequate sense. 

\subsection{Regularity results}
To solve (\ref{eq:System2}) on any time interval
$J_{T}=(0,T),T>0$, we use that for $r\in(1,\infty)$, $\mathcal{L}_{u,v}^{q,q_{\Gamma}}$
has \emph{maximal parabolic regularity} in $L^{r}(J_{T};\mathrm{W}^{-1,q,q_{\Gamma}})$
and that there are $q>d$, $q_{\Gamma}>2$ such that $\mathrm{dom}(\mathcal{L}_{u,v}^{q,q_{\Gamma}}+\lambda)=\mathrm{W}^{1,q,q_{\Gamma}}$
for $\lambda>0$ (for details and proofs, see \cite{Disser17}[Sect. 2.3]). 

For this, if $d=3$, we use the following additional assumption on
$\mu$: it is of the form $\mu(x,u)=\mu_{\circ}(u) m(x)$,
where $\mu_{\circ}\colon\mathbb{R}_{+}\to\mathbb{R}_{+}$ is a scalar
function that is locally Lipschitz and $m\colon\Omega\to\mathbb{R}_{>0,\mathrm{sym}}^{3\times3}$
is continuous (for a slight generalization, see \cite{Disser17}[Sect. 2.2(4)]). 

We define 
\[
\mathrm{MR}_{q,q_{\Gamma}}^{r}:=L^{r}(J_{T};\mathrm{W}^{1,q,q_{\Gamma}})\cap W^{1,r}(J_{T};\mathrm{W}^{-1,q,q_{\Gamma}})
\]
as the space of solutions corresponding to the setting of system (\ref{eq:System2})
and 
\[
\mathrm{X}_{q,q_{\Gamma}}^{r}:=(\mathrm{W}^{1,q,q_{\Gamma}},\mathrm{W}^{-1,q,q_{\Gamma}})_{1-\frac{1}{r},r}
\]
as the corresponding time trace space, satisfying
\[
\mathrm{MR}_{q,q_{\Gamma}}^{r}\hookrightarrow C^{0}(J_{T};\mathrm{X}_{q,q_{\Gamma}}^{r}).
\]
The following embedding result is essential \cite{Disser17}[Lemma 2.4(2)]: If
$q>d$, $q_{\Gamma}>d-1$ and $r>\max(\frac{2q}{q-d},\frac{2q_{\Gamma}}{q_{\Gamma}-d+1})$,
then 
\begin{equation}
\mathrm{X}_{q,q_{\Gamma}}^{r}\hookrightarrow\mathrm{C}^{\beta,\beta_{\Gamma}}\hookrightarrow\mathrm{C}^{0,0},
\end{equation}\label{Xembedding}
for $0<\beta\leq1-\frac{d}{q}-\frac{2}{r}$ and $0<\beta_{\Gamma}\leq1-\frac{d-1}{q_{\Gamma}}-\frac{2}{r}$.
In particular, this implies that for any $(u,v)\in\mathrm{MR}_{q,q_{\Gamma}}^{r}$,
at any $t>0$, the linearized operator $\mathcal{L}_{u(t),v(t)}^{q,q_{\Gamma}}$
is well-defined.

\subsection{Equilibria}
Note that the system (\ref{eq:System1}) preserves the total mass
\[
\mathrm{m}=\beta\int_{\Omega}u(t,x)\,\mathrm{d}x+\alpha\int_{\Gamma}v(t,y)\,\mathrm{d}y.
\]
Given initial values $(u_{0},v_{0})\geq0$ with mass 
\[
\mathrm{m}=\beta\int_{\Omega}u_{0}(x)\,\mathrm{d}x+\alpha\int_{\Gamma}v_{0}(y)\,\mathrm{d}y,
\]
the Wegscheider conditions 
\begin{align}
\mathrm{m} & =\beta\vert\Omega\vert u_{*}+\alpha\vert\Gamma\vert v_{*},\nonumber \\
0 & =u_{*}^{\alpha}-\kappa v_{*}^{\beta},\label{eq:Wegscheider2}
\end{align}
determine uniquely the equilibrium $(u_{*},v_{*})\in\mathbb{R}_{+}^{2}$
corresponding to $(u_{0},v_{0})$ (to see this here, note that the
function $v\mapsto\beta\vert\Omega\vert\kappa^{1/\alpha}v^{\beta/\alpha}+\alpha\vert\Gamma\vert v_{*}$
is strictly monotone and thus invertible from $\mathbb{R}_{+}$ to
$\mathbb{R}_{+}$.) 

\section{Main result\label{sec:Main-result}}
\begin{theorem}
\label{thm:mr} There exist $q>d$, $q_{\Gamma}>2$ such that for
all $r>\max\left(\frac{2q}{q-d},\frac{2q_{\Gamma}}{q_{\Gamma}-d+1}\right)$,
and for all positive $(u_{0},v_{0})\in\mathrm{X}_{q,q_{\Gamma}}^{r}$,
there exists a unique global positive solution $(u,v)\in\mathrm{MR}_{q,q_{\Gamma}}^{r}$
of system (\ref{eq:System2}).
\end{theorem}
The proof is based on the main result in \cite{Disser17} that provides global
existence and uniquenss in the functional analytic framework of system
(\ref{eq:System2}) if the right-hand-side $\mathcal{F}$ is Lipschitz
and preserves $L^{\infty}$-bounds (a posteriori). Regarding regularity
of solutions, note that 
\begin{equation}
\mathrm{MR}_{q,q_{\Gamma}}^{r}\hookrightarrow C^{\delta}(J_{T};\mathrm{C}^{\gamma,\gamma_{\Gamma}})\label{eq:Hoelder}
\end{equation}
for suitable $\delta,\gamma,\gamma_{\Gamma}>0$ \cite{Disser17}[Lemma 2.4(3)]
and note that $\mathrm{W}^{1,q,q_{\Gamma}}\hookrightarrow \mathrm{X}_{q,q_{\Gamma}}^{r}$
by definition (as an upper limit on the regularity required of
the initial data -- with parabolic smoothing, less than $\mathrm{X}_{q,q_{\Gamma}}^{r}$
can be shown to be sufficient). 
\begin{proof}
To be precise, we first note how the assumptions of Theorem 3.1 in
\cite{Disser17} can be satisfied (there may be several choices): in \cite{Disser17}, set $m_{\Gamma}=m_{+}=m_{-}=0$,
$\Omega_{+}=\Omega$, $\Omega_{-}=\emptyset$, $u_{+}=u$, $u_{\Gamma}=v$,
$f_{+}=h_{+}=h_{\Gamma}=0$, $g_{+}(u,v)=-k\alpha(u^{\alpha}-\kappa v^{\beta})$,
$f_{\Gamma}(u,v)=k\beta(u^{\alpha}-\kappa v^{\beta})$, $k_{+}(u)=\mu(u)$,
and $k_{\Gamma}(u,v)=\mu_{\Gamma}(u,v)$. Then the assumptions in
Section 2.1 in \cite{Disser17} are satisfied and it remains to check Assumption 2.5 there
(local Lipschitzianity of $\mathcal{F}$): for $(u_{1},v_{1}),(u_{2},v_{2})\in\mathrm{X}_{q,q^{\Gamma}}^{r}$
with $\Vert(u_{1},v_{1})\Vert_{\mathrm{x}_{q,q_{\Gamma}}^{r}},\Vert(u_{2},v_{2})\Vert_{\mathrm{x}_{q,q_{\Gamma}}^{r}}\leq\tilde{L}$,
we have 
\[
\Vert(u_{1},v_{1})\Vert_{\mathrm{L}^{\infty}},\Vert(u_{2},v_{2})\Vert_{\mathrm{L}^{\infty}}\leq L_{\infty},
\]
by embedding (\ref{Xembedding}), so 
\begin{align*}
\Vert\mathcal{F}(u_{1},v_{1})-\mathcal{F}(u_{2},v_{2})\Vert_{\mathrm{W}^{-1,q,q_{\Gamma}}} & \leq C\Vert u_{1}^{\alpha}-u_{2}^{\alpha}+\kappa(v_{2}^{\beta}-v_{1}^{\beta})\Vert_{L^{\infty}(\Gamma)}\\
 & \leq C(L_{\infty})\Vert(u_{1}-u_{2}),(v_{1}-v_{2})\Vert_{\mathrm{L}^{\infty}}\\
 & \leq C(\tilde{L})\Vert(u_{1}-u_{2}),(v_{1}-v_{2})\Vert_{\mathrm{X}_{q,q_{\Gamma}}^{r}}.
\end{align*}
The global result Theorem \ref{thm:mr} below follows from Theorem 3.1 in \cite{Disser17} as in Corollary
3.6 there, using Lemma \ref{lem:LinftyEst} below on $L^{\infty}$-bounds.
In fact, a Schaefer fixed point argument is used for the proof of
existence of solutions. It requires uniform $L^{\infty}$ -bounds
for solutions of system (\ref{eq:System2}) with right-hand-side $\lambda\mathcal{F}(u,v)$,
$0<\lambda\leq1$ and initial value $\lambda(u_{0},v_{0})$. Clearly,
this also follows from the proof of Lemma (\ref{lem:LinftyEst}).
\end{proof}

\section{$L^{\infty}$-estimates based on the entropic gradient structure\label{sec:-estimates-based-on}}

The main result of this section is the following Lemma. It is sufficient
for proving Theorem \ref{thm:mr}, but also the aim is to introduce
a method that derives pointwise upper bounds from the entropy-producing
structure of the reaction, without using the diffusivity. 
\begin{lemma}
\label{lem:LinftyEst}Let $r,q,q_{\Gamma}$ be as in Theorem \ref{thm:mr}
and let $0<(u_{0},v_{0})\in \mathrm{X}_{q,q_{\Gamma}}^{r}$. Then the corresponding
equilibrium $(u_{*},v_{*})$ is positive and there are constants $c_{u},c_{v},C_{u},C_{v}>0$
such that 
\[
0<c_{u}\leq\frac{u_{0}}{u_{*}}\leq C_{u}<+\infty,\qquad0<c_{v}\leq\frac{v_{0}}{v_{*}}\leq C_{v}<+\infty
\]
for all $x\in\Omega,y\in\Gamma$. Assume that $(u,v)\in \mathrm{MR}_{q,q_{\Gamma}}^{r}$
is a solution of (\ref{eq:System2}) with initial data $(u_{0},v_{0})$.
Define $l:=\min(c_{u}^{\alpha},\kappa c_{v}^{\beta})$ and $L:=\max(C_{u}^{\alpha},C_{v}^{\beta})$,
then 
\begin{equation}
\left(\frac{u(t,x)}{u_{*}}\right)^{\alpha},\left(\frac{v(t,y)}{v_{*}}\right)^{\beta}\leq L,\text{ and }l\leq u^{\alpha}(t,x),\kappa v^{\beta}(t,x),\label{eq:bounds}
\end{equation}
for all $t\geq0$, $x\in\Omega$ and $y\in\Gamma$. 
\end{lemma}
\begin{proof}
First note that since $\mathrm{m}>0$, $(u_{*},v_{*})$ is positive, and 
the embedding (\ref{Xembedding})
implies that $(u_{0},v_{0})$ are positively bounded from above and
below. 

To show that the reaction in (\ref{eq:System1}) is entropy-producing, the reaction rate is reformulated as
\begin{align}
f(u,v): & =k(u^{\alpha}-\kappa v^{\beta})\nonumber \\
 & =\Lambda(u^{\alpha},\kappa v^{\beta})(\alpha\ln u-\beta\ln v-\ln\kappa)\label{eq:PotentialRate}\\
 & =\Lambda(u^{\alpha},\kappa v^{\beta})(\alpha\ln\frac{u}{u_{*}}-\beta\ln\frac{v}{v_{*}}),\nonumber 
\end{align}
where the last equality follows from (\ref{eq:Wegscheider2}). Here,
\[
\Lambda(a,b)=\begin{cases}
\frac{a-b}{\ln a-\ln b}, & \text{if }a,b>0,\text{a\ensuremath{\neq}b},\\
0, & \text{if }a=0\text{ or }b=0,\\
a, & \text{if }a=b>0,
\end{cases}
\]
denotes the logarithmic mean. The equality in (\ref{eq:PotentialRate})
is justified if $u,v\geq0$. Since $u,v\geq0$ will hold a posteriori
(even a uniform positive lower bound), the rate function $f(u,v)=k(u^{\alpha}-\kappa v^{\beta})$
can be replaced with 
\begin{equation}
f_{0}(u,v):=\begin{cases}
f(u,v)=\Lambda(u^{\alpha},\kappa v^{\beta})(\alpha\ln\frac{u}{u_{*}}-\beta\ln\frac{v}{v_{*}}), & \text{if }u,v>0,\\
0, & \text{if }u\leq0\text{ or }v\leq0.
\end{cases}\label{eq:fl}
\end{equation}
We do not know positivity of solutions a priori, so to be precise,
we also use the following very small trick: given $0<(u_{0},v_{0})\in\mathrm{C}^{0,0}$
and the corresponding equilibrium $(u_{*},v_{*})>0$, we show the
bounds (\ref{eq:bounds}) a posteriori, hence, a priori, in (\ref{eq:System2}),
the diffusion coefficients $\mu,\mu_{\Gamma}$ can be replaced by
\[
\mu_{l}^{L}(u) := \mu(\hat{u}^L_l),\quad \mu_{\Gamma,l}^l (u,v) := \mu_\Gamma (\hat{u}^L_l, \hat{v}^L_l),
\]
where
\[
\hat{u}_{l}^{L} :=\begin{cases}
u, & \frac{l}{2}<\left(\frac{u}{u_{*}}\right)^{\alpha}<2L,\\
u_* (\frac{l}{2})^{1/\alpha}, & \left(\frac{u}{u_{*}}\right)^{\alpha}<\frac{l}{2},\\
u_* (2L)^{1/\alpha}, & 2L<\left(\frac{u}{u_{*}}\right)^{\alpha},
\end{cases}
\]
and 
\[
\hat{v}_{l}^{L} :=\begin{cases}
v, & \frac{l}{2}<\left(\frac{v}{v_{*}}\right)^{\alpha}<2L,\\
v_* (\frac{l}{2})^{1/\alpha}, & \left(\frac{v}{v_{*}}\right)^{\alpha}<\frac{l}{2},\\
v_* (2L)^{1/\alpha}, & 2L<\left(\frac{v}{v_{*}}\right)^{\alpha}.
\end{cases}
\]
By the assumptions in Subsection \ref{model}, $\mu_{l}^{L}$ and $\mu_{\Gamma,l}^{L}$
are uniformly bounded and elliptic with 
\begin{align*}
X\cdot\mu_{l}^{L}(x,u)X\geq\underline{\mu}|X|^{2}, & \qquad X\cdot\mu_{l}^{L}(x,u)X'\leq\overline{\mu}\vert X\vert\vert X'\vert,\\
Y\cdot\mu_{\Gamma l}^{L}(y,u,v)Y\geq\underline{\mu}_{\Gamma}|Y|^{2}, & \qquad Y\cdot\mu_{\Gamma,l}^{L}(y,u,v)Y'\leq\overline{\mu}_{\Gamma}\vert Y\vert\vert Y'\vert,
\end{align*}
for a.e. $x\in\Omega$ and $y\in\Gamma$, for all $X,X'\in\mathbb{R}^{d}$
and $Y,Y'\in\mathbb{R}^{d-1}$, and for all $(u,v)\in\mathbb{R}^{2}$
.

Now we actually show that the adapted system (\ref{eq:System2}) with $\mu_l^L, \mu_{\Gamma,l}^L$
has a unique global regular solution $(u,v)$ that satisfies (\ref{eq:bounds}).
But then, 
$$f_{0}(u(t),v(t))=f(u(t),v(t)),\mu_{l}^{L}(u(t))=\mu(u(t)),$$
and 
$$\mu_{\Gamma,l}^{L}(u(t),v(t))=\mu_{\Gamma}(u(t),v(t)),$$ 
for all $t\geq0$. Due to the H\"older regularity of $(u,v)$ in space and time
(\ref{eq:Hoelder}) and due to the local Lipschitzianity of $f$, $\mu$,
and $\mu_{\Gamma}$, $(u,v)$ is also the unique solution of
the original system (\ref{eq:System2}). 

The upper bounds in Lemma \ref{lem:LinftyEst} are shown using the
entropy-producing structure of the system. In the usual proof of a
comparison principle, the dissipation of the $L^{2}$-energy of truncated
solutions is used to show pointwise bounds. Here, the idea is that,
similarly, the dissipation of the free energy of truncated potentials
provides pointwise bounds. 

For this method to work, the solution must be constructed to have an entropy estimate
and it must be sufficiently regular to allow for the truncation.
Here, we are dealing with (fairly) regular solutions and it is straightforward
to do everything rigorously (see (\ref{eq:entropyProduction}) below). Regarding weaker notions of solutions, in \cite{Fischer17}, the corresponding
relative entropy estimate is shown for a general class of reaction-diffusion
systems, for renormalized (and thus for weak) solutions. A similar
result should hold for this simple volume-surface case, even with nonlinear
diffusion (the truncated version may not follow automatically for weak solutions though). In \cite{BFL18}, the
entropy estimate is assumed and used to show the exponential equilibration of the linear diffusion volume-surface
system.

Let $\mathcal{E}$ be the (negative) relative entropy for the system
given by 
\[
\mathcal{E}((u,v);(u_{*},v_{*}))(t)=\int_{\Omega}u_{*}e(\frac{u(t,x)}{u_{*}})\,\mathrm{d}x+\int_{\Gamma}v_{*}e(\frac{v(t,y)}{v_{*}})\,\mathrm{d}y,
\]
where 
\[
e(z)=\begin{cases}
z\ln z-z+1, & z>0,\\
1, & z=0.
\end{cases}
\]
Moreover, let 
\[
\mathcal{E}_{L}((u,v);(u_{*},v_{*}))(t)=L^{1/\alpha}\int_{\Omega}u_{*}e(\frac{u(t,x)}{u_{*}})\,\mathrm{d}x+L^{1/\beta}\int_{\Gamma}v_{*}e(\frac{v(t,y)}{v_{*}})\,\mathrm{d}y,
\]
a relative entropy adapted to the bound $L$. Then define 
\[
u^{L}(t,x)=\begin{cases}
u_{*}, & \left(\frac{u(t,x)}{u_{*}}\right)^{\alpha}\leq L,\\
\frac{u(t,x)}{L^{1/\alpha}}, & \left(\frac{u(t,x)}{u_{*}}\right)^{\alpha}>L,
\end{cases}\quad v^{L}(t,x)=\begin{cases}
v_{*}, & \frac{v(t,x)}{v*}\leq L^{1/\beta},\\
\frac{v(t,x)}{L^{1/\beta}}, & \frac{v(t,x)}{v_{*}}>L^{1/\beta},
\end{cases}
\]
so that
\[
\mathcal{E}_{L}((u^{L},v^{L});(u_{*},v_{*}))(0)=\mathcal{E}_{L}((u_{*},v_{*});(u_{*},v_{*}))=0,
\]
and 
\begin{align*}
 & \frac{d}{dt}\mathcal{E}_{L}((u^{L},v^{L});(u_{*},v_{*}))(t)\\
 = & L^{1/\alpha}\int_{\Omega}\dot{u^{L}}(t,x)e'(\frac{u^{L}(t,x)}{u_{*}})\,\mathrm{d}x+L^{1/\beta}\int_{\Gamma}\dot{v^{L}}(t,y)e'(\frac{v^{L}(t,y)}{v_{*}})\,\mathrm{d}y.
\end{align*}
With the definition 
\[
\xi^{L}(t,x):=e'(\frac{u^{L}(t,x)}{u_{*}})=\begin{cases}
0, & \left(\frac{u(t,x)}{u_{*}}\right)^{\alpha}\leq L,\\
\ln(\frac{u(t,x)}{u_{*}})-\frac{1}{\alpha}\ln L, & \left(\frac{u(t,x)}{u_{*}}\right)^{\alpha}>L,
\end{cases}
\]
and 
\[
\chi^{L}(t,y):=e'(\frac{v^{L}(t,y)}{v_{*}})=\begin{cases}
0, & \frac{v(t,y)}{v_{*}}\leq L^{1/\beta},\\
\ln(\frac{v(t,y)}{v_{*}})-\frac{1}{\beta}\ln L, & \frac{v(t,y)}{v_{*}}>L^{1/\beta},
\end{cases}
\]
this can be written as
\[
\frac{d}{dt}\mathcal{E}_{L}((u^{L},v^{L});(u_{*},v_{*}))(t)=\int_{\Omega}\dot{u}(t,x)\xi^{L}(t,x)\,\mathrm{d}x+\int_{\Gamma}\dot{v}(t,y)\chi^{L}(t,y)\,\mathrm{d}y.
\]
 Due to $(u,v)\in\mathrm{MR}_{q,q_{\Gamma}}^{r}$ and since $e'$
is uniformly Lipschitz on $(u_{*},+\infty)$, $(v_{*},+\infty)$,
\[
(\xi^{L},\chi^{L})\in L^{r}(J_{T};\mathrm{W}^{1,q',q'_{\Gamma}}),
\]
so apply equation (\ref{eq:System2}) to $(\xi^{L},\chi^{L})$ and
add both components to get that, for a.e. $t>0$, 
\begin{align}
\frac{d}{dt}\mathcal{E}_{L}((u^{L},v^{L});(u_{*},v_{*})) & =-\int_{\Omega}\nabla u\cdot\mu_{l}^{L}(u)\nabla\xi^{L}\,\mathrm{d}x-\int_{\Gamma}\nabla_{\Gamma}v\cdot\mu_{\Gamma,l}^{L}(u,v)\nabla\chi^{L}\,\mathrm{d}y\nonumber \\
 & \qquad-\int_{\Gamma_{0}}k\Lambda(u^{\alpha},\kappa v^{\beta})(\alpha\ln\frac{u}{u_{*}}-\beta\ln\frac{v}{v_{*}})(\alpha\xi^{L}-\beta\chi^{L})\,\mathrm{d}y,\nonumber \\
\label{eq:entropyProduction}
\end{align}
where $\Gamma_{0}\subset\Gamma$ is the subset of $\Gamma$, where
$(u,v)\geq0$ . If 
\[
\frac{d}{dt}\mathcal{E}_{L}((u^{L},v^{L});(u_{*},v_{*}))(t)\leq0,
\]
for all $t>0$, then 
\[
\mathcal{E}_{L}((u^{L},v^{L});(u_{*},v_{*}))(t)=\mathcal{E}_{L}((u^{L},v^{L});(u_{*},v_{*}))(0)=0,
\]
hence,
\[
(u^{L},v^{L})(t)=(u_{*},v_{*})
\]
for all $t>0$ ($\mathcal{E}_{L}$ is a strict Lyapunov functional),
and then the upper bound in Lemma \ref{lem:LinftyEst} is proved.
To show this, the terms on the right-hand-side of (\ref{eq:entropyProduction})
can be considered in a pointwise manner. First the reaction term:
for each $t>0$, let 
\begin{align*}
\Gamma_{\xi} & =\{(t,y)\in(0,\infty)\times\Gamma:\xi^{L}(t,y)>0,\chi^{L}(t,y)=0\},\\
\Gamma_{\chi} & =\{(t,y)\in(0,\infty)\times\Gamma:\xi^{L}(t,y)=0,\chi^{L}(t,y)>0\},\\
\Gamma_{\xi,\chi} & =\{(t,y)\in(0,\infty)\times\Gamma:\xi^{L}(t,y),\chi^{L}(t,y)>0\},
\end{align*}
then $\ln\left(\frac{u}{u_{*}}\right)^{\alpha}-\ln\left(\frac{v}{v_{*}}\right)^{\beta}>0$
in $\Gamma_{\xi}$ and $\ln\left(\frac{u}{u_{*}}\right)^{\alpha}-\ln\left(\frac{v}{v_{*}}\right)^{\beta}<0$
in $\Gamma_{\chi}$, so 
\[
(\alpha\ln\frac{u}{u_{*}}-\beta\ln\frac{v}{v_{*}})(\alpha\xi^{L}-\beta\chi^{L})=\begin{cases}
\alpha(\ln\left(\frac{u}{u_{*}}\right)^{\alpha}-\ln\left(\frac{v}{v_{*}}\right)^{\beta})\xi^{L}>0, & \text{in }\Gamma_{\xi},\\
\beta(\ln\left(\frac{v}{v_{*}}\right)^{\beta}-\ln\left(\frac{u}{u_{*}}\right)^{\alpha})\chi^{L}>0, & \text{in }\Gamma_{\chi},\\
(\ln\left(\frac{u}{u_{*}}\right)^{\alpha}-\ln\left(\frac{v}{v_{*}}\right)^{\beta})^{2}>0, & \text{in }\Gamma_{\chi,\xi},
\end{cases}
\]
where in the last line, the adaptation of $\mathcal{E}$ with the
weights $L^{1/\alpha}$, $L^{1/\beta}$ was used. This would not be
necessary if the case that both $u$ and $v$ are pointwise large
can be excluded a priori, an estimate that is often known and employed
in the linear diffusion case, see \cite{Pierre10}. In summary, 
\[
-\int_{\Gamma_{0}}k\Lambda(u^{\alpha},\kappa v^{\beta})(\alpha\ln\frac{u}{u_{*}}-\beta\ln\frac{v}{v_{*}})(\alpha\xi^{L}-\beta\chi^{L})\,\mathrm{d}y\leq0
\]
for the reaction term in (\ref{eq:entropyProduction}).

For the diffusion, the estimate is 
\begin{align*}
-\int_{\Omega}\nabla u\cdot\mu_{l}^{L}(u)\nabla\xi^{L}\,\mathrm{d}x & =-\int_{\Omega_{\xi}=\{x\in\Omega:\left(\frac{u(t,x)}{u_{*}}\right)^{\alpha}\geq L\}}\nabla\sqrt{u}\cdot\mu_{l}^{L}(u)\nabla\sqrt{u}\,\mathrm{d}x\\
 & \leq -\int_{\Omega_{\xi}}\underline{\mu}\vert\nabla\sqrt{u}\vert^{2}\,\mathrm{d}x\leq0,
\end{align*}
and, in the same way,
\[
-\int_{\Gamma}\nabla v\cdot\mu_{l,\Gamma}^{L}(u,v)\nabla\chi^{L}\,\mathrm{d}y\leq -\int_{\Gamma_{\xi}}\underline{\mu}_{\Gamma}\vert\nabla\sqrt{v}\vert^{2}\,\mathrm{d}y\leq0.
\]
This proves the upper bound. The uniform lower bound can be proved by a standard comparison principle
(dissipation for the $L^{2}$-norm): Let $\sigma_{u}:=l^{1/\alpha}$
and $\sigma_{v}:=\kappa l^{1/\beta}$, then $\sigma_{u}\leq u_{0}$
and $\sigma_{v}\leq v_{0}$ and $(\sigma_{u},\sigma_{v})$ is a (stationary)
solution of system (\ref{eq:System1}). Define 
\[
u_{-}(t,x)=\begin{cases}
0, & u(t,x)\geq\sigma_{u},\\
u(t,x)-\sigma_{u}, & u(t,x)>\sigma_{u},
\end{cases}\quad v_{-}(t,x)=\begin{cases}
0, & v(t,y)\geq\sigma_{v},\\
v(t,y)-\sigma_{v}, & v(t,y)>\sigma_{v},
\end{cases}
\]
then $(u_{-},v_{-})(0)=0$ and $(u_{-},v_{-})\in L^{r'}(J_{T};\mathrm{W}^{1,q',q'_{\Gamma}})$
is a suitable test function for (\ref{eq:System2}), so that 
\begin{align*}
 & \frac{\mathrm{d}}{\mathrm{d}t}\left(\int_{\Omega}u_{-}^{2}(t,x)\,\mathrm{d}x+\int_{\Gamma}v_{-}^{2}(t,y)\,\mathrm{d}y\right)+\int_{\Omega}\underline{\mu}\vert\nabla u_{-}\vert^{2}\,\mathrm{d}x+\int_{\Gamma}\underline{\mu}_{\Gamma}\vert\nabla_{\Gamma}v_{-}\vert^{2}\,\mathrm{d}y\\
\leq & -\int_{\Gamma_{0}}k(u^{\alpha}-\kappa v^{\beta})(\alpha u_{-}-\beta v_{-})\,\mathrm{d}y.
\end{align*}
We estimate the reaction term in detail: for each $t>0$, let 
\begin{align*}
\Gamma_{u} & :=\{(t,y)\in(0,\infty)\times\Gamma:u(t,y)<\sigma_{u},v_{-}(t,y)=0\},\\
\Gamma_{v} & :=\{(t,y)\in(0,\infty)\times\Gamma:u_{-}(t,y)=0,v_{-}(t,y)<\sigma_{v}\},\\
\Gamma_{<} & :=\{(t,y)\in(0,\infty)\times\Gamma:u(t,y)<\sigma_{u},v(t,y)<\sigma_{v},u^{\alpha}<\kappa v^{\beta}\},\\
\Gamma_{>} & :=\{(t,y)\in(0,\infty)\times\Gamma:u(t,y)<\sigma_{u},v(t,y)<\sigma_{v},u^{\alpha}>\kappa v^{\beta}\},
\end{align*}
then for some (generic) constant $C>0$, 
\begin{align*}
-\int_{\Gamma_{u}}k(u^{\alpha}-\kappa v^{\beta})(\alpha u_{-}-\beta v_{-})\,\mathrm{d}y & =-\int_{\Gamma_{u}}\alpha k(u^{\alpha}-\sigma_{u}^{\alpha}-\kappa(v^{\beta}-\sigma_{v}^{\beta}))u_{-}\,\mathrm{d}y\leq0,\\
-\int_{\Gamma_{v}}k(u^{\alpha}-\kappa v^{\beta})(\alpha u_{-}-\beta v_{-})\,\mathrm{d}y & \leq-\int_{\Gamma_{v}}\beta\kappa k(v^{\beta}-\sigma_{v}^{\beta})v_{-}\,\mathrm{d}y\leq0,\\
-\int_{\Gamma_{<}}k(u^{\alpha}-\kappa v^{\beta})(\alpha u_{-}-\beta v_{-})\,\mathrm{d}y & \leq\int_{\Gamma_{<}}\beta\kappa k(u^{\alpha}-\sigma_{u}^{\alpha})v_{-}\,\mathrm{d}y,\\
 & \overset{*}{\leq}\int_{\Gamma_{<}}\alpha\beta\kappa k\sigma_{u}^{\alpha-1}u_{-}v_{-}\,\mathrm{dy},\\
 & \leq C\int_{\Gamma_{<}}\vert v_{-}\vert^{2}\,\mathrm{d}y+\frac{\underline{\mu}}{2}\int_{\Omega}\vert\nabla u_{-}\vert^{2}+\vert u_{-}\vert^{2}\,\mathrm{d}x,
\end{align*}
 where $(*)$ follows from the mean value theorem and $\sigma_{u}>u$,
and the last inequality follows from Young's inequality and the trace
estimate $\int_{\Gamma}\vert u_{-}\vert^{2}\,\mathrm{d}y\leq C\Vert u_{-}\Vert_{H^{1}(\Omega)}.$
Similarly, 
\begin{align*}
-\int_{\Gamma_{>}}k(u^{\alpha}-\kappa v^{\beta})(\alpha u_{-}-\beta v_{-})\,\mathrm{d}y & \leq\int_{\Gamma_{>}}\alpha\beta\kappa k\sigma_{v}^{\beta-1}u_{-}v_{-}\,\mathrm{dy}\\
 & \leq C_{>}\int_{\Gamma_{>}}\vert v_{-}\vert^{2}\,\mathrm{d}y+\frac{\underline{\mu}}{2}\int_{\Omega}\vert\nabla u_{-}\vert^{2}+\vert u_{-}\vert^{2}\,\mathrm{d}x.
\end{align*}
Hence, 
\[
\frac{\mathrm{d}}{\mathrm{d}t}\left(\int_{\Omega}u_{-}^{2}(t,x)\,\mathrm{d}x+\int_{\Gamma}v_{-}^{2}(t,y)\,\mathrm{d}y\right)\leq C\left(\int_{\Omega}u_{-}^{2}(t,x)\,\mathrm{d}x+\int_{\Gamma}v_{-}^{2}(t,y)\,\mathrm{d}y\right)
\]
 and thus, by Gronwall's inequality, $(\sigma_{u},\sigma_{v})\leq(u,v)(t)$
for all $t\geq0$. 
\end{proof}



\end{document}